\documentclass[12pt]{amsart}
\usepackage{latexsym,amssymb,amsmath,amsthm}
\usepackage{comment}
\usepackage{listings}
\usepackage{float}
\usepackage{graphicx}
\newtheorem{thm}{Theorem}
\newtheorem{lemma}[thm]{Lemma}
\newtheorem{propo}[thm]{Proposition}
\newtheorem{coro}[thm]{Corollary}
\theoremstyle{definition}

\newcommand{\C}{\mathbb{C}} \newcommand{\R}{\mathbb{R}}  
\def\G{\mathcal{G}}  \def\L{\mathcal{L}} 
\def\tr{\rm{tr}}     \def\Det{\rm{Det}}
 \def\Tr{{\rm Tr}}
\def\rank{{\rm rank}} 
\def\det{{\rm det}}

\title{Cauchy-Binet for Pseudo-Determinants}
\author{Oliver Knill}
\date{June 17, 2014}
\address{
        Department of Mathematics \\
        Harvard University \\
        Cambridge, MA, 02138
        }

\begin{document}
\maketitle

\begin{abstract}
The pseudo-determinant $\Det(A)$ of a square matrix $A$ is defined 
as the product of the nonzero eigenvalues of $A$. It is a basis-independent number which 
is up to a sign the first nonzero entry of the characteristic polynomial of $A$. 
We prove $\Det(F^T G) = \sum_P \det(F_P) \det(G_P)$ for any two $n \times m$ matrices $F,G$. 
The sum to the right runs over all $k \times k$ minors of $A$, where $k$ is determined by $F$ and $G$.
If $F=G$ is the incidence matrix of a graph this directly implies the Kirchhoff tree theorem
as $L=F^T G$ is then the Laplacian and $\det^2(F_P) \in \{0,1\}$ is equal to $1$ if $P$ is a 
rooted spanning tree.
A consequence is the following Pythagorean theorem: for any self-adjoint matrix $A$ of rank $k$, 
one has $\Det^2(A) = \sum_P \det^2(A_P)$, where $\det(A_P)$ runs over $k \times k$ minors
of $A$. More generally, we prove the polynomial identity 
$\det(1+x F^T G) = \sum_P x^{|P|} \det(F_P) \det(G_P)$
for classical determinants $\det$, which holds for any two $n \times m$ matrices $F,G$ and where
the sum on the right is taken over all minors $P$, understanding the sum to be $1$ if $|P|=0$.
It implies the Pythagorean identity  
$\det(1+F^T F) = \sum_P \det^2(F_P)$ which holds for any $n \times m$ matrix $F$ and sums again over
all minors $F_P$. If applied to 
the incidence matrix $F$ of a finite simple graph, it produces the Chebotarev-Shamis
forest theorem telling that $\det(1+L)$ is the number of rooted spanning forests in the 
graph with Laplacian $L$. 
\end{abstract}

\section{Introduction}

The {\bf Cauchy-Binet theorem} for two $n \times m$ matrices $F,G$ with $n \geq m$ tells that
\begin{equation}
   \det(F^T G) = \sum_{P} \det(F_P) \det(G_P) \; ,
   \label{classicalcauchybinet}
\end{equation}
where the sum is over all $m \times m$ square sub matrices $P$ and $F_P$
is the matrix $F$ masked by the pattern $P$. In other words, $F_P$ is an $m \times m$ matrix obtained 
by deleting $n-m$ rows in $F$ and $\det(F_P)$ is a minor of $F$. 
In the special case $m=n$, the formula is the product formula $\det(F^T G) = \det(F^T) \det(G)$ 
for determinants. Direct proofs can be found in \cite{MarcusMinc,LancasterTismenetsky,Morrow}. 
An elegant multi-linear proof is given in \cite{Shafarevich}, where it is called ``almost tautological".
A graph theoretical proof using the Lindstr\"om-Gessel-Viennot lemma 
sees matrix multiplication as a concatenation process of directed graphs and determinants 
as a sum of weighted path integrals \cite{AigZie}.
The classical Cauchy-Binet theorem implies the Pythagorean identity 
$\det(F^T F) = \sum_P \det^2(F_P)$ for square matrices
which is also called Lagrange Identity \cite{HuppertWillems},
where $P$ runs over all $m \times m$ sub-matrices of $F$. This formula is used in multivariable calculus
in the form $|\vec{v}|^2 |\vec{w}|^2 - (\vec{v} \cdot \vec{w})^2 = |\vec{v} \wedge \vec{w}|^2$
and is useful to count the number of basis choices in matroids \cite{Aigner1979}. 
The Cauchy-Binet formula assures that the determinant is compatible with the matrix product. 
Historically, after Leibniz introduced determinants in 1693, and Van der Monde made it into a theory
in 1776 \cite{Kline2}, Binet and Cauchy independently found
the product formula for the determinant around 1812 \cite{Binet,Cauchy,Pascal,BourbakiHistory,Schneider},
even before matrix multiplication had been formalized. \cite{Coolidge,Kline2} noticed that Lagrange 
had mentioned a similar result before in the three dimensional
case. The term ``matrix" was used by Sylvester first in 1850 \cite{Sylvester,Sylvester1850a,Sylvester1850b}. 
The book \cite{Kline2} mentions that 
Binet's proof was not complete. It was Cayley who looked first at the matrix algebra 
\cite{Cayley1843,Knobloch,Muir}. 
It is evident today that the Cauchy-Binet formula played a pivotal role for the development 
of matrix algebra. Early textbooks like \cite{Reiss} illustrate how much the notation of 
determinants has changed over time. \\

In this paper, we extend the Cauchy-Binet formula~(\ref{classicalcauchybinet})
to matrices with determinant $0$. Theorem~\ref{newcauchybinet}) will follow from Theorem~\ref{Theorem2}) 
and imply new formulas for classical determinants like Theorem~\ref{determinanttheorem}). 
The pseudo-determinant $\Det(A)$ for a square matrix $A$ is defined
as the product of the nonzero eigenvalues of $A$, with the assumption $\Det(0)=1$
for a matrix with all zero eigenvalues like $A=0$ or nilpotent matrices. 
The later assumption renders all formulas also true for zero
matrices.  Looking at singular matrices with pseudo-determinants
opens a new world, which formula (\ref{classicalcauchybinet}) buries under the 
trivial identity ``$0=0$''.
The extension of Cauchy-Binet to pseudo-determinants is fascinating because these determinants 
are not much explored and because Cauchy-Binet for pseudo-determinants is not a trivial 
extension of the classical theorem. One reason is that the most commonly used form of Cauchy-Binet 
is false, even for diagonal square matrices $A,B$: 
while $\Det(A B) = \Det(B A)$ is always true, we have in general:
$$  \Det(A B) \neq  \Det(A) \Det(B) \;   $$
as the example $A=\left[ \begin{array}{cc} 2 & 0 \\ 0 & 0 \\ \end{array} \right]$,
   $B=\left[ \begin{array}{cc} 0 & 0 \\ 0 & 2 \\ \end{array} \right]$ shows.  \\

What can be generalized? Because eigenvalues of square matrices $C$ and $C^T$
agree, it is true that $\Det(C)=\Det(C^T)$ for square matrices. In particular, 
$\Det(F^T G) = \Det(G^T F)$ if $F,G$ are $n \times m$ matrices.
It is also true - even so it is slightly less obvious - that 
$\Det(F^T G) = \Det(F G^T)$. The later follows from the fact that $F^T G$ and $F G^T$ are
essentially isospectral, meaning that they have the same nonzero eigenvalues. 
If $F,G$ are not square, then one of the products has zero eigenvalues
so that we need the pseudo-determinant for this identity to be interesting. Experiments
indicated us that summing over determinants of square matrices on the right hand side works. 
We will show:

\begin{thm}
If $F,G$ are $n \times m$ matrices, then
\begin{equation}
   \Det(F^T G) = \sum_{P} \det(F_{P}) \det(G_{P}) \; ,
   \label{newcauchybinet}
\end{equation}
where the sum is over all $k \times k$ sub matrix masks $P$ of $F$,
and where $k$ is such that $p_{k} (-x)^{m-k}$ is the smallest order
entry in the characteristic polynomial 
$p_0 (-x)^m +p_1 (-x)^{m-1} + \cdots + p_k (-x)^{m-k}$ of $F^T G$.
\label{Theorem1}
\end{thm}

The formula holds also in limiting cases like for nilpotent $F^T G$, 
in which case $k=0$ and both sides are $1$. \\

This is more general than the classical Cauchy-Binet theorem. If $F^T G$ is not invertible
but $F G^T$ is, then one can use that $\Det(F^T G) = \det(F G^T)$ and use the classical
Cauchy-Binet result. But Theorem~\ref{Theorem1}) also applies if both $F^T G$ and $F G^T$ 
are singular. For a self-adjoint square matrix $A$ of rank $k$ in particular, 
the formula shows that $\Det(A^2)$ is a sum
of squares of determinants $\det^2(B)$  of $k \times k$ sub matrices $B$.
This uses that $\Det^2(A) = \Det(A^2)$, which is one of the
identities for normal matrices to be discussed in the next section.\\

The pairing $\langle F,G \rangle = \Det(F^T G) = \langle G,F \rangle$ leads to a functional
$||F|| = \sqrt{ \langle F,F \rangle }$ on $n \times m$ matrices. It is far from a norm, as
$||0||=1$, and no triangle inequality nor Cauchy-Schwarz 
inequality $\langle F,G \rangle \leq ||F|| ||G||$ holds.
The sum of determinants $\langle F,G \rangle$ depends in a nonlinear way on both $F$ and $G$:
if $\lambda \neq 0$, then 
$$ \langle \lambda F,G \rangle = \lambda^k \langle F,G \rangle  
                               = \langle F, \lambda G \rangle \; , $$
where $k$ is the number of nonzero eigenvalues of $F^T G$ and where we again understand 
that for $k=0$, all terms are $1$. \\

The ``sphere" $X=||A||=1$ in $M(n,\R)$ is a union of Cauchy-Binet varieties 
$X_n={\rm SL}(n,R),X_{n-1}, \dots  X_1,X_0$. In the case $M(2,\R)$ for example, we have 
$X=X_2 \cup X_1 \cup X_0 = {\rm SL}(2,\R) \cup \{ ad-bc=0,a+d=1 \; \} \cup N$, where $X_0=N$ is the
set of nilpotent matrices. The set $X_1= \{ a(1-b)-b c=1 \; \}$ is a two-dimensional quadric. 
In the case of diagonal $2 \times 2$ matrices, we have 
$X=X_2 \cup X_1 \cup X_0  = \{ |a d|=1 \; \} \cup \{ a d=0, |a+d|=1 \; \} \cup \{a=d=0\}$ 
which is a union of a hyperbola, four points $\{(\pm 1,0) \;\},\{(0, \pm 1) \;\}$ and the origin $\{(0,0)\;\}$.
In the case $M(3,R)$ already, the unit sphere $X$ is the $8$-dimensional 
$X_3 = {\rm SL}(3,R)$ together with a $7$-dimensional $X_2$ and a $6$-dimensional variety $X_1$
as well as the variety of nilpotent matrices. 
The classical Cauchy-Binet theorem misses the three later ones. 
We see that the case $\det(A)=0$ is an unexpectedly rich place. \\

One of the main motivations for pseudo-determinants is graph theory, where
the Laplacian matrix $L$ always has a kernel. While $\det(L)$ is zero and therefore
not interesting, $\Det(L)$ has combinatorial meaning as it allows to count
spanning trees in the graph. The number $\Det(L)$ indeed is a measure for the 
complexity of the graph. This paper grew while looking at questions in graph theory.
The classical Kirchhoff matrix tree theorem is 
an immediate consequence of Theorem~\ref{Theorem1}) because if $F=G$ 
is the incidence matrix of a graph then $A=F^T G$ is the scalar Laplacian and 
$\Det(A) = \Det(F^T G) = \sum_P \det(F_P)^2$. 
The Kirchhoff matrix tree theorem \cite{Biggs} can rely on the classical
Cauchy-Binet theorem for invertible matrices. The reason is that for a connected graph, the kernel
of the Laplacian is one-dimensional, so that $\Det(A)= n \cdot \det(M)$, 
where $M$ is a minor of $A$ which is a classical determinant.
The proof can then proceed with the classical Cauchy-Binet theorem for $M$. 
We are forced to pseudo determinants when looking at geometric interpretations of
trees in the simplex graph $\G$, which is the graph consisting of all complete subgraphs of
a graph where two simplices of dimension difference $1$ are connected if one is contained
in the other. The operator $L=D^2$ on discrete differential forms \cite{knillmckeansinger} will 
then play an important role as the matrix $|D|$ is the adjacency matrix of $\G$. 
The matrix $L$ has a large kernel in general with 
${\rm dim}({\rm ker}(L)) = \sum_i b_i=b$, the sum of the Betti numbers of the graph which by
the Hodge theorem is the total dimension of all harmonic forms. Theorem~\ref{Theorem1}) gives
a combinatorial interpretation of what $\Det(L)$ means, if $L$ is the form Laplacian
of a graph. The matrix $L$ is large in general: if $G=(V,E)$ has $v$ complete subgraphs, 
then $\G$ has $v$ vertices and $L$ is a $v \times v$ matrix. 
For the complete graph $K_n$ for example, $v=2^n-1$. We will see below that also the 
classical Kirchhoff matrix theorem can be understood more naturally using Cauchy-Binet for
pseudo determinants. \\

{\bf Acknowledgements}:
Thanks to {\bf Fuzhen Zhang} for a pointer on generalized matrix functions,
to {\bf Nieves Castro Gonz\'alez} for references like \cite{Bapat,MiaoBenIsrael} and 
{\bf Woong Kook} for information about more sophisticated matrix tree theorems. 
Thanks to {\bf Shlomo Sternberg} for a historical discussion on Binet in 2007 which sparked my interest
in this part of the history of linear algebra. Added in proof is a reference to \cite{Knillforest},
where it pointed out that the results of the current paper leads to a proof of the 
Chebotarev-Shamis forest theorem \cite{ChebotarevShamis2,ChebotarevShamis1}. 
Some linear algebra in this context has been summarized
in the talk \cite{KnillILAS}.

\section{The pseudo-determinant} 

The {\bf pseudo-determinant} of a square $n \times n$ matrix $A$ is 
defined as the product of the nonzero eigenvalues of $A$ with the convention
that the pseudo-determinant of a nilpotent matrix is $1$. The choice for the 
later makes all formulas hold true unconditionally. As the sum over the empty set is
naturally zero, the product over the empty set is naturally assumed to be $1$. 
With this definition, the pseudo determinant is then the smallest 
coefficient $p_k$ of the characteristic polynomial $p_A(x) = \det(A-x) = p_0 (-x)^n
+p_1 (-x)^{n-1} + \cdots + p_k (-x)^{n-k} + \cdots + p_n$. 
This is $1$ for nilpotent $A$, as $\det(A-x)=(-x)^n$ for such matrices. An other extreme
is $\Det(A)=p_n=\det(A)$ if the matrix $A$ is invertible. \\

We start with some basic facts about the pseudo-determinant $\Det(A)$ 
of a $n \times n$ matrix $A$. Most points are obvious, but we did not find any references. 
Some multi-linear treatment of the pseudo-determinant will appear in the 
proof section of Theorem~\ref{Theorem2}). \\

We denote with $A^*$ the adjoint of $A$ and with 
$A^+$ the Moore-Penrose pseudo-inverse of $A$ defined by 
$A^+=V D^+ U^*$ if $A=UDV^*$ is the singular value decomposition (SVD) of $A$
and $D^+$ is the diagonal matrix which has the same zero entries than $D$
and where $D^{+}_{ii}=1/D_{ii}$ for the nonzero diagonal entries of $D$. Pseudo-inverses
appear textbooks like \cite{Strang2009} in the context of SVD. While SVD
is not unique, the Moore-Penrose pseudo-inverse is. 
Denote by $p_A(x) = \det(A-x)$ the characteristic polynomial of $A$.
We use a sign choice used from \cite{Strang2009,Bretscher} and computer algebra
systems like Mathematica. A matrix is self-adjoint if $A=A^*$, it is normal if 
$A A^* = A^* A$. We also denote by $Q$ the unit cube in $\R^n$ and by $|Y|_k$ the $k$-volume
of a $k$-dimensional parallelepiped $Y$ in $\R^n$. Let $\Lambda^k A$ denote the $k$'th
exterior power of $A$. It is a $\left( \begin{array}{c} n \\ k \end{array} \right)
\times \left( \begin{array}{c} n \\ k \end{array} \right)$ matrix which is determined by 
$$ \Lambda^k A (e_1 \wedge \cdots \wedge e_k) = (Ae_1 \wedge \cdots \wedge Ae_k)  \; . $$
If a basis $\{ e_i \}$ in $\R^n$ is given, then  $e_I = e_{i_1} \wedge \cdots \wedge e_{i_k}$
with $I \subset \{1, \dots ,n\}$ of cardinality $k$ defines a basis of $\Lambda^k \R^n$. 
The matrix entries of $[\Lambda^k A]_{IJ}$ are $(\Lambda^k A e_I) \cdot e_J$. 
The Pfaffian of a real skew-symmetric $2n \times 2n$ matrix $A$ is defined as follows:
since $A$ is naturally a $2$-form in $\R^{2n}$, one can look at its $n$'th exterior power
$A \wedge \cdots \wedge A$ which is a $2n$ form ${\rm pf}(A) e_1 \wedge \cdots \wedge e_{2n}$ and
consequently defines the real number ${\rm pf}(A)$. 
Since the Pfaffian satisfies ${\rm pf}^2(A)=\det(A)$, we can define the pseudo Pfaffian ${\rm Pf}(A)$
of a real skew-symmetric matrix $A$ as the product of the nonzero entries $\lambda_i$, if $A$ 
is orthogonally conjugated to a block-diagonal matrix containing diagonal blocks 
$B_j=\left[ \begin{array}{cc} 0 & \lambda_j \\ -\lambda_j & 0 \end{array} \right]$. The block 
diagonalization follows from the spectral theorem for real skew-symmetric matrices, which 
in turn follows from the fact that $iA$ is a complex
selfadjoint matrix if $A$ is real skew symmetric and that 
$S=\left[ \begin{array}{cc} 1 & -1 \\ i & i \end{array} \right]$ conjugates the skew diagonal
blocks $B_j$ to  $S^{-1} B_j S 
= \left[ \begin{array}{cc} i \lambda_j & 0 \\ 0 & -i \lambda_j \end{array} \right]$. 
We understand again that the pseudo Pfaffian satisfies ${\rm Pf}(A)=1$ if $A=0$. 
While ${\rm pf}(A)=0$ for all $(2n+1) \times (2n+1)$ matrices, the pseudo Pfaffian ${\rm Pf}(A)$ is never zero. 

\begin{propo}
Assume $A,B \in M(n,\C)$. \\
{\bf 1)} If $A$ is similar to $B$ then $\Det(A) = \Det(B)$. \\
{\bf 2)} If $A$ is invertible, then $\Det(A)=\det(A)$. \\
{\bf 3)} If $P$ is an orthogonal projection onto a subspace then $\Det(P)=1$. \\
{\bf 4)} If $A$ is real and normal, then $|\Det(A)| = |A(Q)|_k$ if $k={\rm ran}(A)$.  \\
{\bf 5)} $\Det(A^T) = \Det(A)$ and $\Det(A^*)=\overline{\Det(A)}$. \\
{\bf 6)} $\Det$ is discontinuous on $M(n,\R)$ or $M(n,\C)$ for $n \geq 2$. \\
{\bf 7)} The pseudo-inverse of a normal matrix satisfies $\Det(A^+)=1/\Det(A)$. \\
{\bf 8)} $\Det(A) = p_k$, where $p_A(x) = \det(A-x) = p_0 (-x)^n + \cdots + p_k (-x)^k$. \\
{\bf 9)} $\Det(A) = \tr(\Lambda^k A)$, if $A$ has exactly $k$ nonzero eigenvalues. \\
{\bf 10)} For any $A$ and $m \in {\bf N}$, one has $\Det(A^m) = \Det^m(A)$. \\
{\bf 11)} For $A,B \in M(n,R)$ then $\Det(A^T B) = \Det(A B^T)$. \\
{\bf 12)} $\Det(A)$ is never zero and if $A=A^*$, then $\Det(A)$ is real. \\
{\bf 13)} For block diagonal $A={\rm Diag}(A_1,\dots,A_k)$, $\Det(A)=\prod_i \Det(A_i)$.  \\
{\bf 14)} For real skew $A$, the pseudo Pfaffian satisfies ${\rm Pf}^2(A)=\Det(A)$. 
\label{Proposition}
\end{propo}

\begin{proof}
{\bf 1)} The eigenvalues of similar matrices are the same. \\
{\bf 2)} We use the definition and the fact that the classical determinant is the product of the eigenvalues. \\
{\bf 3)} If we diagonalize $P$, we get a matrix with only $1$ or $0$ in the diagonal. In the case of $P=0$
         the result holds also by definition. \\
{\bf 4)} $\Det(A)$ is basis independent and self-adjoint matrices can be diagonalized.
         For an orthogonal projection in particular, the pseudo-determinant is $1$. \\
{\bf 5)} The eigenvalues are the same for $A$ and $A^T$ and 
         $A^*$ has the complex conjugate eigenvalues of $A$. \\
{\bf 6)} The functional is already discontinuous for $n=2$, where  \\
         $\Det(\left[ \begin{array}{cc} a & 0 \\ 0 & 2 \end{array} \right]) = 2a$ 
         for $a \neq 0$ and $\Det(\left[ \begin{array}{cc} a & 0 \\ 0 & 2 \end{array} \right])=2$ for $a=0$. \\
{\bf 7)} Use that the pseudo inverse has the nonzero eigenvalues $\lambda_j^{-1}$ if $\lambda_j$ 
         are the eigenvalues of $A$.  For non normal matrices this can be false: 
         the matrix $A=\left[ \begin{array}{cc} 1 & 1 \\ 0 & 0 \end{array} \right]$ has the eigenvalues $0,1$.
         Its pseudo inverse $A^+ = \left[ \begin{array}{cc} 1/2 & 0 \\ 1/2 & 0 \end{array} \right]$
         has eigenvalues $0,1/2$. \\
{\bf 8)} Write $p_A(x) = \prod_j (\lambda_j-x) $, where $\lambda_j$ runs over the set of nonzero eigenvalues. \\
{\bf 9)} This follows from the previous step and the fact that $\tr(\Lambda^k A) = p_k$, a fact which 
         can be deduced from $\det(1+A) = \sum_{j=0}^n \tr(\Lambda^j A))$ (see i.e. \cite{RSIV} p.322), an
         identity which allows to define the determinant $\det(1+A)$ in some infinite dimensional setups.
         The result holds also in the nilpotent case $k=0$ as $\Lambda^0(A)$ is the $1 \times 1$ matrix $1$.  \\
{\bf 10)} Normal matrices $A$ can be diagonalized over the field of complex numbers: there exists a unitary $U$ such that $U^* A U$ is diagonal. 
          In general, we can bring the matrix in Jordan normal form. The result is also true for $n=0$, where both sides are $1$.  \\
{\bf 11)} The matrices $A^T B$ and $A B^T$ have the same nonzero eigenvalues
          because their characteristic polynomials differ by a factor $\lambda^k$ only. \\
{\bf 12)} The first statement is by definition and can also be seen from 8) as the characteristic polynomial is
          never zero. For $A=A^*$, then the eigenvalues of $A$ are real by the spectral theorem.  \\
{\bf 13)} Group the eigenvalues of $A$ as eigenvalues of $A_i$. Also here, the formula works if some of the $A_i$ are nilpotent. \\
{\bf 14)} $A$ is orthogonally conjugated to a block diagonal matrix $B$. Now square this and use part $1)$. 
\end{proof} 

{\bf Remarks.} \\
{\bf 1)} We can compute pseudo-determinants almost as fast as determinants because we only need to know the 
characteristic polynomial. Some computer algebra code can be found in the Appendix.
We can find $\Det(A)$ also by row reduction if we do safe row reduction steps. 
As mentioned below, we have to make sure that we do not count any sign changes when swapping two parallel 
rows and do scalings of dependent rows, nor subtract a row from a parallel row. 
When doing safe row reductions, we end up with a matrix which looks like a 
row reduced echelon matrix but where parallel rows can appear. For such a reduced matrix, the eigenvalues can 
be computed fast.  \\
{\bf 2)} The least square solution formula $Ax=A (A^T A)^{-1} A^T y = P y$ is pivotal in statistics and features a projection 
matrix $P$ with pseudo-determinant $1$. We mention this because the least square inverse is often also called pseudo inverse
even so it has nothing to do with the Moore-Penrose pseudo inverse in general.  The former deals with overdetermined systems 
$Ax=y$, the later is defined for square matrices $A$ only. \\
{\bf 3)} Define $\log^+|x| = \log|x|$ for $x \neq 0$ and $\log^+|0|=0$. 
If $A$ is normal, we can define $\log^+|A|$ by diagonalizing $A$ and get
$$   \tr(\log^+|A|) = \log|\Det(A)| \; , $$
Also this is consistent with nilpotent $A$ and an other reason
why the definition $\Det(A)=1$ for nilpotent $A$ is natural. \\
{\bf 4)} For finite simple graphs, the Laplace operator $L$ always has a kernel. It is one-dimensional if the 
graph is connected. The pseudo-determinant is considered a measure for complexity because it allows to count the
number of maximal spanning trees in the graph. The Laplace-Beltrami operator on forms has a large kernel in general. 
Its dimension is the sum of the Betti numbers of the graph. Studying this matrix associated to a graph was the 
main motivation for us to look at Cauchy-Binet in the singular case.  \\
{\bf 5)} The fact that $F^T G$ and $F G^T$ have the same nonzero eigenvalues is a consequence of 
$p_A(x) = (-x)^k p_B(x))$ if $A=F^T G$ and $B=F G^T$.
For $x=-1$, this leads to ``the most important identity in mathematics (Deift)" \cite{Tao2012} 
$\det(1+F G)=\det(1+G F)$ for $n \times m$ and $m \times n$ matrices $F,G$. \cite{Deift1978} illustrates how rich such 
identities can be. Switching operators is useful to construct the spectrum of the quantum harmonic oscillator,
for the numerical QR algorithm $A=QR \to RQ$ used to diagonalize a matrix, or to construct new solutions 
to nonlinear PDE's with so called B\"acklund transformations. \\

Lets look at some pitfalls: \\

{\bf 1)} One could think that if $B$ is invertible, then 
$\Det(A B) = \Det(A) \det(B)$. This is false: a counter example is 
$A=\left[ \begin{array}{cc} 1 & 0 \\ 0 & 0 \end{array} \right]$
and $B=\left[ \begin{array}{cc} 1 & 0 \\ 0 & 2 \end{array} \right]$.  It is even false for unitary $B$ like
$A=\left[ \begin{array}{cc} 1 & 1 \\ 1 & 1 \end{array} \right]$ and 
$B=\left[ \begin{array}{cc} 0 & 1 \\ 1 & 0 \end{array} \right]$,
where $AB=A$ and $\det(B)=-1$. This example is the row swapping pitfall using the elementary swapping matrix $B$. 
We mention this pitfall because we actually tried to prove this first by repeating the textbook proof of 
$\det(A B) = \det(A) \det(B)$ in which one makes row reduction on the augmented matrix $[A|B]$ until 
the matrix $B$ is the identity matrix. As we see below, pseudo-determinants need safe row reduction operations. \\
{\bf 2)} Even basic multi-linearity fails, even if we apply it to nonzero rows. 
An example like $A=\left[ \begin{array}{cc} 1 & 1 \\ 1 & 1 \\ \end{array} \right]$
shows that if we scale a row by $\lambda$ then the pseudo-determinant gets scaled 
by $(1+\lambda)/2$ and not by $\lambda$. This is just an example and not a general rule. It is difficult to 
say in general how scaling a linearly dependent row affects the pseudo-determinant. \\
{\bf 3)} The case of block diagonal matrices $A={\rm Diag}(A_1,A_2)$ with square matrices 
$A_1,A_2$ again illustrates that it is useful to define $\Det(A)=1$ for nilpotent matrices. 
This renders $\det({\rm Diag}(A_1,A_2)) = \Det(A_1) \cdot \Det(A_2)=0$ true. Without the assumption like 
assuming $\Det(0)=0$, this block diagonal result would be wrong. \\
{\bf 4)} While $F^T G$ and $F G^T$ are essentially isospectral and therefore 
have the same number $k$ of nonzero eigenvalues, it is not true in 
general that ${\rm rank}(F^T G)$ and ${\rm rank}(F G^T)$ are the same. It is possible 
that $k$ is smaller than both: let $F$ be the identity matrix $I_3$ and
let $G=\left[ \begin{array}{ccc}1&-1&-1\\1&-1&-1\\-2&2&-1\\ \end{array} \right]$, then 
$p_G(x) = -x^2-x^3$ so that we have to take $k=1$ in Theorem~\ref{Theorem1}) 
and $\Det(F^T G) = -1$. But ${\rm rank}(F^T G)={\rm rank}(F G^T)=2$.

\section{Examples} 

{\bf 1)} An extreme example of Theorem~\ref{Theorem1}) is obtained when $F,G$ are two column vectors in $\R^n$
and where the left hand side is the ``dot product" $F^T G$. It agrees with 
$$                  \Det \left[ \begin{array}{cccc} 
                                      F_1 G_1 & F_1 G_2 & \cdots & F_1 G_n \\
                                      F_2 G_1 & F_2 G_2 & \cdots & F_2 G_n \\
                                       \cdots & \cdots  & \cdots & \cdots  \\
                                      F_n G_1 & F_n G_2 & \cdots & F_n G_n 
                                      \end{array} \right]  \;  $$
which can be seen as the pseudo determinant analogue of a {\bf Gram determinant} \cite{GohbergGoldberg}
and which has the only non-zero eigenvalue $F^T G$. 
While the characteristic polynomial of the $1 \times 1$ matrix
$A=F^T G$ is $p_A(\lambda) = A-\lambda$, the characteristic polynomial of 
the $n \times n$ matrix $F G^T$ is $(-\lambda)^{n-1}(A-\lambda)$.  \\

{\bf 2)} Assume $F$ is a $n \times m$ matrix for which every row $v_j$ of $F$ is a multiple $a_j$ of $v$ and
that $G$ is a $n \times m$ matrix for which every row $w_j$ of $G$ is a multiple $b_j$ of a vector $w$. Then
$\Det(F^T G) =  (v \cdot w)  (a \cdot b)$. This is the same than the right hand side of Cauchy-Binet.
Since there is only nonzero eigenvalue, it has to be the trace of $F^T G$.
The later matrix has the entries $(v \cdot w) a_i b_j$.
The left hand side of Cauchy-Binet is $(\sum_{i} v_i w_i ) (\sum_j a_j w_j)$.
The right hand side of Cauchy-Binet is $\sum_{i,j} (a_j v_i)  (b_j w_i)$. These two sums
are the same.  The same works also if $F$ is rank $1$ and $G$ is arbitrary.
Since $F^T G$ has rank $1$, we have $\Det(F^T G) = \tr(F^T G) = \sum_{i,j} F_{ij} G_{ij}$.  \\

{\bf 3)} The two $3 \times 2$ matrices 
$$ F = \left[ \begin{array}{cc} 1 & 4 \\ 
                                2 & 5 \\ 
                                3 & 6 \end{array} \right] \; , 
   G = \left[ \begin{array}{cc} 1 & 0 \\ 
                                1 & 1 \\ 
                                1 & 0 \end{array} \right]  \;  $$
give
$$ F^T G = \left[ \begin{array}{cc} 6 & 2 \\ 
                                   15 & 5 \\ \end{array} \right],
   F G^T = \left[ \begin{array}{ccc} 1 & 5 & 1 \\ 
                                     2 & 7 & 2 \\
                                     3 & 9 & 3 \end{array} \right] \; . $$
Both are singular matrices with the same determinant $\Det(F^T G) = \Det(F G^T) = 11$. 
If $F,G$ are $n \times 2$ matrices with column vectors $a,b$ and $c,d$, then the 
classical Cauchy-Binet identity is 
$$ (a \cdot c) (b \cdot d) - (a \cdot d) (b \cdot c) = \sum_{1 \leq i,j \leq n} (a_i b_j-a_j b_i) (c_i d_j - c_j d_i) \;  $$
which has the special case $(a \cdot c)^2 - |a|^2 |c|^2 = |a \times c|^2$ in three dimensions for the cross product 
which expresses the Pythagoras theorem $\cos^2(\alpha) - 1 = \sin^2(\alpha)$. 
Assume now that one of the matrices has rank $1$ like in the case $a=c$. The classical identity is then $0$ on both 
sides. The new Cauchy-Binet identity is not interesting in this case: 
since $F^T G$ has the eigenvalues $0,a(\sum_i b_i + \sum_i c_i)$, 
we have $\Det(F^T G) = a(\sum_i b_i + \sum_i c_i)$. The right hand side is 
$\sum_{|P|=1} \det(F_P) \det(G_P) =  \sum_{i} a b_i + a c_i$. \\

{\bf 4)} For
$$ F=\left[ \begin{array}{ccc} 3 & 1 & 1 \\ 2 & 2 & 2 \\ \end{array} \right],
   G=\left[ \begin{array}{ccc} 0 & 0 & 2 \\ 0 & 3 & 2 \\ \end{array} \right] $$
we have 
$$ F^T G = \left[ \begin{array}{ccc} 0 & 6 & 10 \\ 0 & 6 & 6 \\ 0 & 6 & 6 \\ \end{array} \right], 
   F G^T = \left[ \begin{array}{cc} 2 & 5 \\ 4 & 10 \\ \end{array} \right]  \; , $$
which both have the pseudo-determinant $12$. Now, the $2 \times 2$ square sub-matrices of $F$ and $G$ are 
$$ F_1 = \left[ \begin{array}{cc} 3 & 1 \\ 2 & 2 \\ \end{array} \right],
   F_2 = \left[ \begin{array}{cc} 3 & 1 \\ 2 & 2 \\ \end{array} \right],
   F_3 = \left[ \begin{array}{cc} 1 & 1 \\ 2 & 2 \\ \end{array} \right] $$
$$ G_1 = \left[ \begin{array}{cc} 0 & 0 \\ 0 & 3 \\ \end{array} \right],
   G_2 = \left[ \begin{array}{cc} 0 & 2 \\ 0 & 2 \\ \end{array} \right],
   G_3 = \left[ \begin{array}{cc} 0 & 2 \\ 3 & 2 \\ \end{array} \right] \; . $$
But all their products
$$ F_1^T \cdot G_1 = \left[ \begin{array}{cc} 0 & 6 \\ 0 & 6 \\ \end{array} \right] \; ,
   F_2^T \cdot G_2 = \left[ \begin{array}{cc} 0 & 10 \\ 0 & 6 \\ \end{array} \right] \; , 
   F_3^T \cdot G_3 = \left[ \begin{array}{cc} 6 & 6 \\ 6 & 6 \\ \end{array} \right]  \;    $$
have determinant $0$. The reason is that while $\rank(F)=\rank(G)=\rank(F^T G)=2$,
we have $\rank(F G^T)=1$. We see that we have to take the sum over the products $\det(F_P^T G_P)$ where
$P$ runs over all $1 \times 1$ matrices. And indeed, now the sum is $12$ too. This example
shows that even so $F^T G$ and $F G^T$ can have different rank,
their pseudo-determinants are still the same. Of course, this follows in general from the fact
that all the nonzero eigenvalues are the same. \\

{\bf 5)} For non-invertible $2 \times 2$ matrices $A$, we have $\Det(A) = \tr(A)$. For example,
$$ \Det( \left[ \begin{array}{cc} 5 & 6 \\ 10 & 12 \end{array} \right] ) = 17 \; . $$
The same is true for matrices for which all columns are parallel like 
$$ \Det( \left[ \begin{array}{ccccc} 1 & 2 & 3 & 4  \\
                                     1 & 2 & 3 & 4  \\
                                     1 & 2 & 3 & 4  \\
                                     1 & 2 & 3 & 4  \end{array} \right] ) = \tr(A) = 10 \; . $$
{\bf 6)}
The two matrices 
$$ F = \left[ \begin{array}{cc} 2 & 5 \\
                                0 & 0 \\
                                0 & 0  \end{array} \right], 
   G = \left[ \begin{array}{cc} 3 & 7 \\
                                0 & 1 \\
                                0 & 0  \end{array} \right] $$
have different rank. The determinant of $F^T G = \left[ \begin{array}{cc} 6 & 14 \\ 15 & 35 \end{array}  \right]$
is equal to its trace $41$, as seen in the previous example. 
This agrees with the sum $\sum_{|P|=1}\det(F_P) \det(G_P) = 2 \cdot 3 + 5 \cdot 7 = 41$.  \\

{\bf 7)} If $A$ is a square van der Monde matrix defined by $n$ numbers $a_i$,
then the determinant of $A$ is $\prod_{i<j} (a_i-a_j)$. If the numbers are not different, like for
$$ A = \left[ \begin{array}{ccc} 1 & 1 & 1\\
                                a & b & a \\
                              a^2 & b^2 & a^2 \end{array} \right] \; , $$
we have $\Det(A) = (b-a)(b+a+a b)$ if $a \neq b$ and $\Det(A) = \tr(A) = 1+a+a^2$ if $a=b$.  \\

{\bf 8)} As defined, the result also holds for cases like the nilpotent 
$F=\left[ \begin{array}{cc} -1 & -1 \\ 1 & 1 \\ \end{array} \right]$ and
$G=\left[ \begin{array}{cc}  1 &  0 \\ 0 & 1 \\ \end{array} \right]$ in which case $\Det(F^T G) = \Det(F)=1$
and $\sum_{|P|=0} \det(F_P) \det(G_P)=1$. 

\section{Consequences} 

Theorem~\ref{Theorem1}) implies a symmetry for the nonlinear pairing $\langle F,G \rangle = \Det(F^T G)$: 

\begin{coro}[Duality]
If $F,G$ are matrices of the same shape, then $\Det(F^T G) = \Det(F G^T)$.
\end{coro}

\begin{proof} 
The characteristic polynomials of $F^T G$ and $F G^T$ satisfy $p_{F^T G}(x) = \pm x^l p_{F G^T}(x)$ 
for some integer $l$.
\end{proof} 

{\bf Remarks:} \\
{\bf 1)} The matrices $F^T G$ and $F G^T$ have in general different shapes. The result is 
also true in the square matrix case with the usual determinant because both sides are then $\det(F) \det(G)$. \\
{\bf 2)} We know that $M(n,m)$ with inner product $\tr(F^* G)$ is a Hilbert
space with the Hilbert-Schmidt product $\sum_{i,j} \overline{F_{ij}} G_{ij}$ which is sometimes 
called Frobenius product. 
While the pairing $\langle A,B \rangle = \Det(A^T B)$ is not an inner product,
it inspires to define $||A||^2 := \Det(A^T A) > 0$. It satisfies
$\langle A,B \rangle=\langle B,A \rangle$, 
$d(A,B)=||A-B||$ is not a metric as he triangle inequality does not hold and $||0||^2=1$. 

\begin{coro}[Kirchhoff]
For a connected finite simple graph with Laplacian $L$ and $m$ vertices, $\Det(L)/m$ is the
number of spanning trees in the graph.
\label{Kirchhoff}
\end{coro} 
\begin{proof}
If the graph $(V,E)$ has $|V|=m$ vertices and $|E|=n$ edges, then 
the incidence matrix $F$ is a $n \times m$ matrix. This matrix depends 
on a given orientation of the edges and is a discrete
gradient which assigns to a function on vertices a function on edges
with $F f( (a,b) ) = f(b)-f(a)$. Its adjoint $F^T$ is a discrete divergence.
The combination $L = F^T F = {\rm div}( {\rm grad}( f))$ is the scalar Laplacian and agrees
with the matrix $B-A$, where $B$ is the diagonal vertex degree matrix and $A$ is the adjacency matrix.
Unlike $F$, the matrix $L$ does not depend on the chosen edge orientation. 
A connected graph $L$ has only one eigenvalue $0$ with the constant eigenvector $1$. 
The matrix $L$ therefore has rank $k=n-1$. Choosing a pattern $P$ means to choose a vertex 
$q$ as a root and to select $m-1$ edges. If $F_P$ has determinant $1$ or $-1$, it represents a connected 
tree rooted at $q$ because we have $m-1$ edges and $m$ vertices.
Theorem~\ref{Theorem1}) equals $\Det(L)$ and counts the rooted spanning trees and consequently,
$\Det(L)/m$ is the number of spanning trees in the graph. 
\end{proof}

Here is an obvious corollary of Theorem~\ref{Theorem1}): 

\begin{coro}
If $A$ is any matrix of rank $k$, then $\Det(A^T A) = \Det(A A^T) = \sum_P \det^2(A_P)$,
where $A_P$ runs over all $k \times k$ minors. 
\label{coro1}
\end{coro}
\begin{proof}
This is a special case of Theorem~\ref{Theorem1}),
where $F=G=A$ using the fact that the two matrices $A^T A$ and 
$A A^T$ have the same rank $k$, if $A$ has rank $k$. 
\end{proof} 

Especially, we have a Pythagoras theorem for pseudo-determinants: 

\begin{coro}[Pythagoras]
For a selfadjoint matrix $A$ of rank $k$, then
$$  \Det^2(A) = \Det(A^2) = \sum_P \det^2(A_P) \; , $$
where $A_P$ runs over all $k \times k$ minors of $A$. 
\label{Pythagoras}
\end{coro} 

\begin{proof}
Use that $\Det^2(A) = \Det(A^2) = \Det(A A^T)$ if $A$ is selfadjoint,
as a consequence of 10) in Proposition~\ref{Proposition}).
\end{proof} 

{\bf Remarks.} \\
{\bf 1)} If $A$ is invertible, \cite{KABush} uses a special case of this identity
that if $A$ is a $(n-1) \times n$ matrix, then $\det(A A^T) = n S$ where $S$ is a 
perfect square.  \\
{\bf 2)} As in the classical case, this result can be interpreted geometrically:
the square of the $k$-volume of the parallelepiped spanned by the columns of $A$ is related to 
the squares of the volumes of projections of $A Q_I$ onto planes spanned by $Q_J$ 
where for $I=\{i_1, \dots, i_k \;\}$ the set $Q_I$ is the parallelepiped spanned by 
$e_{i_1}, \dots, e_{i_k}$ and the subsets $I,J$ of $\{1,\dots, n \;\}$ encode
a $k \times k$ sub mask $P$ of $A$. (See \cite{Morrow}). \\
{\bf 3)} Theorem~\ref{Pythagoras}) is obvious for diagonal matrices. An alternative  
proof could be to show that the right hand side is basis independent. But this exactly 
needs to go through the multi-linear approach which sees the right hand side as
a Hilbert-Schmidt norm of a $k$-Fermion matrix. \\

{\bf Examples.} \\

{\bf 1)} If $A=\left[ \begin{array}{ll} a & b \\ a & b  \end{array} \right]$, then 
$\Det(A^T A) = 2a^2+2b^2$ and agrees with the right hand side of Corollary~\ref{Pythagoras}). \\

{\bf 2)} If $A$ is invertible, the Pythagoras formula is trivial and tells  $\det(A^2) = \det^2(A)$.
If $A$ has everywhere the entry $a$ then the left hand side
is $(n a)^2$ and the right hand side adds up $n^2$ determinant squares $a^2$ of 
$1 \times 1$ matrices. \\

{\bf 3)} Corollary~\ref{Pythagoras}) can fail for non-symmetric matrices. The simplest case is 
$A=\left[ \begin{array}{cc} 1 & 2 \\ 0 & 0 \\ \end{array} \right]$, for which 
$\Det(A)=1$ but $\Det(A^T A) = \sum_{|P|=1} \det^2(A_P)=5$. The correct formula for 
non-selfadjoint matrices is $\Det(A^T A) = \sum_P \det^2(A_P)$ as in Corollary~\ref{coro1}). 
This example illustrates again that $\Det(A^T A) \neq \Det^2(A)$ in general. \\

{\bf 4)} The matrix
$$ A= \left[
                  \begin{array}{ccc}
                   0 & 4 & 4 \\
                   4 & 0 & 3 \\
                   4 & 3 & 6 \\
                  \end{array}
                  \right]  $$
has the characteristic polynomial $41 x + 6 x^2-x^3$ and has the eigenvalues 
$0,(3 \pm 5 \sqrt{2})$. The pseudo-determinant is $-41$ so that $\Det^2(A)=1681$. 
Now let's look at the determinants of all the $2 \times 2$ sub-matrices of $A$. 
\begin{center}\begin{tabular}{lll}
$A_1=\left[ \begin{array}{cc} 0 & 4 \\ 4 & 0 \\ \end{array} \right]$ &
$A_2=\left[ \begin{array}{cc} 0 & 4 \\ 4 & 3 \\ \end{array} \right]$ &
$A_3=\left[ \begin{array}{cc} 4 & 4 \\ 0 & 3 \\ \end{array} \right]$ \\ 
                                       & &                           \\ 
$A_4=\left[ \begin{array}{cc} 0 & 4 \\ 4 & 3 \\ \end{array} \right]$ &
$A_5=\left[ \begin{array}{cc} 0 & 4 \\ 4 & 6 \\ \end{array} \right]$ &
$A_6=\left[ \begin{array}{cc} 4 & 4 \\ 3 & 6 \\ \end{array} \right]$ \\ 
                                       & &                           \\ 
$A_7=\left[ \begin{array}{cc} 4 & 0 \\ 4 & 3 \\ \end{array} \right]$ &
$A_8=\left[ \begin{array}{cc} 4 & 3 \\ 4 & 6 \\ \end{array} \right]$ &
$A_9=\left[ \begin{array}{cc} 0 & 3 \\ 3 & 6 \\ \end{array} \right]$ \\
                                       & &                           \; .  \\ 
\end{tabular}\end{center} 
The determinant squares $256, 256, 144, 256, 256, 144, 144, 144, 81$ add up to $1681$.  \\

{\bf 5)} The matrix 
$$ \left[
                   \begin{array}{cccccc}
                    0 & 0 & 0 & 1 & 0 & 0 \\
                    0 & 0 & 1 & 0 & 0 & 1 \\
                    0 & 1 & 0 & 0 & 0 & 0 \\
                    1 & 0 & 0 & 0 & 1 & 0 \\
                    0 & 0 & 0 & 1 & 0 & 0 \\
                    0 & 1 & 0 & 0 & 0 & 0 \\
                   \end{array}
                   \right] $$
has $\Det(A)=4$. From the $15^2 = 225$ sub-matrices of $A$ of size $4 \times 4$, 
there are 16 which have nonzero determinant. Each of them either has 
determinant $1$ or $-1$. The sum of the determinants squared is $16$.  \\

{\bf 6)} $F=\left[ \begin{array}{cccc} 1 & 2 & 3 & 4 \end{array} \right]$ then 
$$ A= F^T F = \left[
                  \begin{array}{cccc}
                   1 & 2 & 3 & 4 \\
                   2 & 4 & 6 & 8 \\
                   3 & 6 & 9 & 12 \\
                   4 & 8 & 12 & 16
                  \end{array}
                  \right] \; $$ 
is selfadjoint with $(\Det(A))^2 =30^2 = 900$. Since $k=1$, the right hand side of
Pythagoras is the sum of the squares of the matrix entries. This is 
also $900$. This can be generalized to any row vector $F$ for which the identity reduces
to the obvious identity $(\sum_{i=1}^n F_i^2)^2 = (\sum_{i,j=1}^n (F_i F_j)^2)$
which tells that the Euclidean norm of $F$ is the Hilbert-Schmidt norm of $A=F^T F$. 

\section{A generalization} 

We have the Hilbert-Schmidt identity
$$ p_1= \tr(F^T G) = \sum_{|P|=1} \det(F_P) \det(G_P) = \sum_{i,j} F_{ij} G_{ij}   \; , $$
where $F_P$ is the sub-matrix matched by the pattern $P$ and 
$|P|=k$ means that $P$ is a $k \times k$ matrix and where
\begin{equation}
 p(x) = p_0 (-x)^m + p_1 (-x)^{m-1} + \cdots + p_k (-x)^{m-k} + \cdots + p_m   
 \label{charpol1} 
\end{equation}
is the characteristic polynomial of the $m \times m$ matrix $F^T G$ with $p_0=1$.
Having seen an identity holds for $k=1$ (and see that it is the classical Cauchy-Binet
for $k=m$) we can try for general $0 \leq k \leq m$. 
Experiments prompted to generalize the result. 
Theorem~\ref{Theorem1}) is the special case of the following, when $k$
is the minimal $k$ for which $p_k$ is not zero: 

\begin{thm}[Generalized Cauchy-Binet]
If $F,G$ are arbitrary $n \times m$ matrices and $0 \leq k$, then 
$$ p_k = \sum_{|P|=k} \det(F_P) \det(G_P)  \; , $$
where the sum is over $k$-minors and where $p_k$ are the coefficients
of the characteristic polynomial (\ref{charpol1}) of $F^T G$. This implies
the polynomial identity
$$ \det(1+z F^T G) = \sum_P z^{|P|} \det(F_P) \det(G_P)  \;  $$
in which the sum is over all minors $A_P$ including the empty one $|P|=0$ for which 
$\det(F_P) \det(G_P)=1$.
\label{Theorem2}
\end{thm}

We will prove this in the next section. The second statement follows
then from the fact that for $z=(-x)^{-1}$ and any $m \times m$ matrix $F^TG$, then 
$$ \det(F^T G-x) = \sum_P (-x)^{m-|P|} \det(F_P) \det(G_P)  \; . $$
Note again that the polynomial identity in the theorem holds also for $z=0$, 
where both sides are $1$ as $|P|=0$ contributes $1$ to the sum. \\

In the next corollary, we again sum over all minors, including the 
empty set $P$ for which $\det(F_P)=1$:

\begin{coro}
If $F,G$ are two general $n \times m$ matrices, then 
$$ \det(1+F^T G) = \sum_P \det(F_P) \det(G_P)  \; ,  $$
where the right hand side sums over all minors.
\label{determinanttheorem}
\end{coro}
\begin{proof}
Just evaluate the polynomial identity in Theorem~\ref{Theorem2}) at $z=1$.
\end{proof} 

Especially,

\begin{coro}
If $F$ is an arbitrary $n \times m$ matrix, then 
$$ \det(1+F^T F) = \sum_P \det^2(F_P) \;  $$
where the right hand side sums over all minors.
\label{determinantcorollary}
\end{coro}

{\bf Examples:} \\
{\bf 1)} If $A=\left[ \begin{array}{cc} a & b \\ c &  d \\ \end{array} \right]$ 
then $\det(1+A^T A) = a^2 d^2+a^2-2 a b c d+b^2 c^2+b^2+c^2+d^2+1$ which is equal to
$1+a^2+b^2+c^2+d^2 + (ad-bc)^2$. \\
{\bf 2)} If $v=[v_1,\dots,v_n]$ is a $n \times 1$ matrix, then $1+v^T \cdot v = 1+v_k^2$. \\
{\bf 3)} If $v=[a,b]^T$ is a $1 \times 2$ matrix we have  $\det( \left[ \begin{array}{cc} 1+a^2 & ab \\
                                                                                    ab    & 1+b^2  \end{array} \right]
= 1+a^2+b^2$. \\
{\bf 4)} If $A$ is a symmetric $n \times n$ matrix then $\det(1+A^2) = \sum_P \det^2(A_P)$.
This is useful if $A=D=(d+d^*)$ is the Dirac matrix of a graph and $A^2=L$ is the form Laplacian.
Since $A_P$ can have several nonzero permutation patterns, the interpretation using rooted
spanning trees is a bit more evolved, using double covers of the simplex graph. \\
{\bf 5)} If $F$ is the incidence matrix for a graph so that $F^T F$ is the scalar Laplacian
then $\det(F_P)^2 \in \{1,0 \}$ and choosing the same number of edges and vertices in such a way 
that every edge connects with exactly one vertex and so that we do not form loops. These are rooted forests, 
collections of rooted trees. Trees with one single vertex are seeds which when included, 
lead to rooted spanning forests. 

\begin{coro}[Chebotarev-Shamis]
$\det(1+L)$ is the number of rooted spanning 
forests in a connected finite simple graph $G$ with scalar Laplacian $L$. 
\end{coro} 
\begin{proof}
Again, we assume that the graph $(V,E)$ has $|V|=m$ vertices and $|E|=n$ edges so that
the incidence matrix $F$ is a $n \times m$ matrix. The condition $m \leq n$ follows from the
connectedness. The Laplacian $L$ is of the form $L=F^T F$ where $F$ is gradient defined by 
Poincar\'e and $F^T$ the divergence. Corollary~\ref{determinantcorollary}) equates $\det(1+L)$ with 
the sum over all possible determinants $\det^2(F_P) \in \{0,1\}$. Now if $|P|=m$, then $\det(F_P)=0$
as $L$ has a kernel. If $|P|=m-1$ we count all rooted spanning trees. For $|P|=n-2$ we count all 
cycle free subgraphs with $n-2$ edges. These are rooted spanning forests as individual trees are no
more necessarily connected and even can be seeds which are graphs without edges. Since a forest with 
$n-k$ edges has Euler characteristic $|V|-|E| = n-(n-k) = k = b_0-b_1 = b_0$ by Euler Poincar\'e, where
$b_0$ is the number of connectivity components, we see that a spanning forest with $n-k$ edges has
$k$ base points. These are the roots of the individual trees. The sum also includes the case $k=n$,
where all trees are seeds and which corresponds to the case $|P|=0$ in the sum. 
\end{proof} 

See \cite{ChebotarevShamis2,ChebotarevShamis1} for the discovery and 
\cite{Knillforest}, for more details and examples. 

\section{Proof}

While the proof of Theorem~\ref{Theorem2}) 
uses standard multi-linear algebra, we need to fix some notation. \\

Recall that $\Lambda^k F$ is the $k$'th exterior power of a $n \times m$ matrix 
$F$. It is a linear transformation from the vector space $\Lambda^k \R^m$ to the
vector space $\Lambda^k \R^n$. The linear transformation $\Lambda^k F$ can be written as a
$\left( \begin{array}{c} n \\ k \end{array} \right) \times \left( \begin{array}{c} m \\ k \end{array} \right)$
matrix if a basis is given in $\R^m$ and $\R^n$. If $k=0$, then $\lambda^0 F$ is just the scalar $1$.  \\

While the matrix entries $F_{ij}$ are defined for $1 \leq i \leq n, 1 \leq j \leq m$,
the indices of $\Lambda^k F$ are given by subsets $I$ of $\{1, \dots, n \; \}$
and subsets $J$ of $\{1, \dots, m \; \}$ which both have cardinality $k$. 
If $e_j$ are basis elements in $\R^m$ and $f_j$ are
basis elements in $\R^n$, then $e_I = e_{i_1} \wedge  \cdots \wedge e_{i_k}$
and $f_J = f_{i_1} \wedge  \cdots \wedge f_{i_k}$ are basis vectors in $\Lambda^k \R^m$
and $\Lambda^k \R^n$. Write $\Lambda^k F e_I$ for $F e_{i_1} \wedge \cdots \wedge F e_{i_k}$.
It is possible to  simplify notation and just write 
$F_{IJ}$ for the matrix entry of $(\Lambda^k F)_{IJ}$. It is a real number given by 
$F e_I \cdot f_J$. Define also $F_P=F_{P(IJ)}$ for the matrix with pattern $P=P(IJ)$ defined 
by the sets $I \subset \{1, \dots, n \; \}$ and $J \subset \{1, \dots, m \; \}$. 
Finally, in the case of a square matrix $A$,
write $\Tr(A) = \sum_{K} A_{KK}$ when summing over all subsets $K$.  \\

The following lemma tells that $F_{IJ}$ is a minor

\begin{lemma}[Minor]
If $\Lambda^k \R^n$ and $\Lambda^k \R^m$ are equipped with the standard basis 
$e_I = e_{i_1} \wedge \cdots \wedge e_{i_k}$ obtained from a basis $e_i$, then 
$$   F_{IJ} = \det(F_{P(IJ)}) \;  $$
for any sets $I,J$ of the same cardinality $|I|=|J|=k$. 
\label{multilinearlemma}
\end{lemma}
\begin{proof}
This is just rewriting $F_{IJ} = \langle F e_I, f_J \rangle$
using the basis vectors $e_I = e_{i_1} \wedge \cdots \wedge e_{i_k}$
and $f_J = f_{j_1} \wedge \cdots \wedge f_{j_k}$ 
and the definition of $\Lambda^k F: \Lambda^k \R^m \to \Lambda^k \R^n$. 
\end{proof} 

Lemma~\ref{multilinearlemma}) implies the trace identity:

\begin{coro}[Trace]
For any $n \times n$ square matrix $A$ and $0 \leq k \leq n$ we have 
$$ \tr(\Lambda^k A) = \sum_{P} \det(A_P) \; ,  $$
where the sum is over all $k \times k$ sub matrices $A_P$.
\label{trace}
\end{coro}

Note that unlike in Cauchy-Binet, we have just summed over 
diagonal minors $A_{P(II)}$. In the case $k=0$, we understand the left hand side 
to be the $1 \times 1$ matrix $\Lambda^0 A=1$ and the right hand side
to sum over the empty pattern $P$ which is $1$ as we postulated 
$\det(F_P)=1$ for $|P|=0$. 


\begin{lemma}[Multiplication]
For $n \times m$ matrix $F$ and an $m \times n$ matrix $G$ 
$$  (F G)_{IJ} = \sum_K F_{IK} G_{KJ}  \;  $$
where $I,J \subset \{1, \dots n\}$ of cardinality $k$ 
and where the sum is over all $K \subset \{1, \dots, m\}$ of cardinality $k$. 
\label{multiplication}
\end{lemma}
\begin{proof}
This lemma appears as Theorem~A.2.1) in \cite{GohbergLancasterRodman}. 
As pointed out in \cite{Shafarevich} Lemma~10.16), it rephrases the 
classical matrix multiplication of the two exterior products
$\Lambda^k A, \Lambda^k B$ as a composition of maps 
$\Lambda^k \R^n \to \Lambda^k \R^m \to \Lambda^k \R^n$.
For more on the exterior algebra, see \cite{AMR,Greub}.
\end{proof} 

We can now prove Theorem~\ref{Theorem2}):

\begin{proof} 
We use $p_k(A) = \tr(\Lambda^k A)$ for any square matrix $A$, Lemma~\ref{multilinearlemma})
and Lemma~\ref{multiplication}): 
\begin{eqnarray*}
        p_k(F^TG) &=& \tr(\Lambda^k (F^T G) )  = \tr( \Lambda^k F^T \Lambda^k G)  \\
                  &=& \sum_J ( \sum_I F^T_{JI} G_{IJ} ) \\
                  &=& \sum_J \sum_I F_{IJ} G_{IJ} \\
                  &=& \sum_J \sum_I \det(F_{P(IJ)}) \det(G_{P(IJ)})  \\
                  &=& \sum_P \det(F_P) \det(G_P)  \; . 
\end{eqnarray*} 
\end{proof}

Theorem~\ref{Theorem2}) implies 

\begin{coro}[Pythagoras]
For any self-adjoint $n \times n$ matrix $A$ and any $0 \leq k \leq n$ we have 
$$  p_k =  \tr(\Lambda^k A^2) = \sum_{|P|=k} \det^2(A_P) \; ,  $$
where the sum is taken over all minors $A_{P(IJ)}$ with $|I|=k,|J|=k$ and
where $p_k$ is a coefficient of the characteristic polynomial 
$p_0 (-x)^n + p_1 (-x)^{n-1} + \cdots + p_n$ of $A^2$. 
\label{pythagoras2}
\end{coro} 

{\bf Remarks.} \\
{\bf 1)} Despite the simplicity of the proof and similar looking results for minor expansion, formulas in
multi-linear algebra \cite{Konstantopoulos, HuppertWillems}, condensation formulas, trace ideals \cite{SimonTrace},
matrix tree results \cite{Biggs},
formulas for the characteristic polynomial \cite{Lewin,Pennisi,Brooks}, pseudo inverses,
non-commutative generalizations \cite{Caracciolo}, we are not aware that even the
special case of Corollary~\ref{Pythagoras}) for the pseudo determinant
has appeared anywhere already.
In the classical case, where $A$ is invertible, the Pythagorean identity is
also called  Lagrange Identity \cite{HuppertWillems}.\\
{\bf 2)} Pythagoras~(\ref{pythagoras2}) should be compared with the {\bf trace identity} given in 
Lemma~\ref{trace}). The former deals with a product of matrices and is therefore a ``quadratic" identity.
The trace identity deals with one matrix only and does not explain Pythagoras yet. 
For $k=1$, the Cauchy-Binet formula and the trace identity reduce to the definition of the matrix 
multiplication and the definition of the
trace: $\tr(F^T G) = \sum_{i,j} F_{ij} G_{ij}$ and $\tr(A) = \sum_i A_{ii}$. If $k$ is the rank of $A$,
then Cauchy-Binet is Theorem~\ref{Theorem1}) and the trace identity is the
known formula $\Det(A) = \tr(\Lambda^k A)$, where $k$ is the rank of $A$.

\section{Row reduction}

One can try to prove Theorem~\ref{Theorem1}) by simplifying both sides of
$$  \Det(F^T G) = \sum_P \det(F_P) \det(G_P) \; , $$
by applying row operations on $F$ and $G$ and
using that both sides of the identity are basis independent. We will only illustrate this here
as a proof along those line is probably much longer. This section might
have some pedagogical merit however and also traces an earlier proof attempt. 
The strategy of row reduction is traditionally used in the proof of Cauchy-Binet. 
In the special case of the product identity $\det(A B) = \det(A) \det(B)$ 
already, there is a proof which row reduces the $n \times 2n$ matrix $[A|B]$. 
This strategy does not generalize, as any of the three row reduction steps are false
for pseudo-determinants. Lets explain the difficulty: classical row reduction of a matrix $A$ consists of applying 
swap, scale or subtract operations to $A$ to bring a $n \times m$ matrix into row reduced echelon form. 
Exercises in textbooks like \cite{Lay,Bretscher}) ask to prove that the end result ${\rm rref}(A)$ is
independent of the strategy with which these steps are applied.  
When applying row reduction to a nonsingular matrix until the identity 
matrix is obtained, then $\det(A) = (-1)^r/\prod \lambda_j$, where $r$ is the number of 
swap operations used and $\lambda_j$ are the scaling constants which were applied during the 
elimination process. While this is all fine for $\det$, for $\Det$ it is simply false.
Multiplying a zero row with $\lambda$ or swapping two zero rows does not alter
the matrix and does therefore not contribute to the $(-1)^r$ or scaling factors. For example, swapping 
the two rows of 
$$ \left[ \begin{array}{cc} 1 & 1 \\ 0 & 0  \\ \end{array} \right] $$
does not change the pseudo-determinant, nor does a multiplication of the second row by $\lambda$. 
As far as the ``subtraction" part of row reduction, there are more bad news:
unlike for the determinant, subtracting a row from an other
row can change the pseudo-determinant. For example,
$$   \Det( \left[ \begin{array}{cc} 1 & 1 \\ 1 & 1  \\ \end{array} \right] ) = 2,
     \Det( \left[ \begin{array}{cc} 1 & 1 \\ 0 & 0  \\ \end{array} \right] ) = 1 \; . $$
These problems can be overcome partly with a
safe Gauss elimination which honors a ``Pauli exclusion principle": we can do row reduction, as long as we 
do not deal with pairs of parallel vectors. This analogy leads to the use of alternating 
multi-linear forms to get a geometric characterization of the pseudo-determinant. Lets assume for simplicity
that $A$ is a square matrix. In this multi-linear setup, $A$ acts on $k$-forms $f=f_1 \wedge f_2 \wedge \dots \wedge f_k$ as 
$Af = Af_1 \wedge Af_2 \wedge \dots \wedge Af_k$. 
The length of the vector $f=f_1 \wedge \cdots \wedge f_k$ is defined as the $k$-volume $|f|$ of the parallelepiped 
spanned by the $k$ vectors $f_1, \dots,f_k$. We have $|f|=\det(F)$, where $F$ is the matrix 
$F_{ij} = f_i \cdot f_j$ and where $f \cdot g$ denotes the Euclidean dot product. 
Lets call a row in $A$ independent within $A$, if it is not parallel to any other nonzero row.
Safe row reduction consists of three steps $A \to B$ using the following rules:
{\bf a)} if an independent row of $A$ is scaled by a factor $\lambda \neq 0$, then $\Det(B) = \lambda \Det(A)$,
{\bf b)} if two independent rows in $A$ are swapped, then $\Det(B) = -\Det(A)$,
{\bf c)} if an independent row is subtracted from an other independent row, then $\Det(B) = \Det(A)$.
We would like to apply this to the pseudo-determinant of $A=F^T G$: 
lets see what happens if we append a multiple of a given row, starting with the assumption that all of rows are already 
independent. While it is difficult to predict the effect of adding a multiple of a row to an 
other row, it is possible to see when what happens to $A = F^T G$ if such an operation is performed
for $F$ and for $G$.  \\
{\bf i)} Appending a parallel row. \\
Given two $n \times m$ matrices $F,G$ such that $F^T G$ is nonsingular. 
Assume $A^T$ is the $n \times (m+1)$ matrix obtained from $F^T$ by appending 
$\lambda $ times the $l$'th row of $F$ at the end. Assume that $B^T$ is the
$n \times (m+1)$ matrix obtained from $G^T$ by appending $\mu$ times the $l$'th row
of $G^T$ at the end. Then both sides of the Cauchy-Binet formula are multiplied by 
$1+\lambda \mu$. \\
{\bf Proof. } First show that $\det(A^T B) = \Det(A^T B)  = (1+\lambda \mu)  \Det(F^T G)$.
First bring the $m$ row vectors in row reduced echelon form so that we end 
up both for $F$ and $G$ with matrices which are row reduced in the first $m-1$ rows and for which the $m$'th
and $m+1$ th row are parallel. If we now reduce the last two rows,  $F^T G$ is block diagonal with 
$\left[ \begin{array}{cc} 1 & \mu \\ \lambda & \lambda+\mu \end{array} \right]$
at the end which has pseudo-determinant $1 + \lambda \mu$. 
For every pattern $P$ which does not involve the $l$'th row, we have $\det(F_P) = \det(A_P)$
and $\det(G_P) = \det(B_P)$. For every pattern $P$ which does involve the $l$' row and 
not the second last we have $\det(A_P) = \lambda \det(F_P)$ and 
$\det(B_P) = \mu \det(G_P)$. For every pattern which involves the appended last row as well as the
$l$'th row, we have $\det(A_P)=\det(B_P)=0$.  \\
{\bf ii)} Given two $n \times m$ matrices $F,G$ such that $F^T G$ and $F G^T$ have maximal rank. Given $1 \leq l \leq m$. 
Assume the row $v=\sum_{j=1}^l \lambda_j v_j$ is appended to $F^T$ and $w=\sum_{j=1}^l \mu_j w_j$ is appended 
to $G^T$, where $v_j$ are the rows of $F^T$ and $w_j$ are the rows of $G^T$.  Then both sides of the 
Cauchy-Binet formula are multiplied by $1 + \sum_{j=1}^l \lambda_j \mu_j$.  \\
{\bf Proof.} Use induction with respect to $l$, where $l=1$ was case (i). 
When adding a new vector $l \to l+1$, the determinant increases by $\lambda_{l+1} \mu_{j+1}$. 
On the right hand side this is clear by looking at the patterns which involve the last $(m+1)$'th row. \\
When adding more rows we do not have explicit formulas any more.  \\

{\bf 1)} This example hopes to illustrate, how difficult it can be to 
predict what scaling of a row does to the pseudo-determinant.
Lets take the rank-2 matrix
$A= \left[ \begin{array}{ccc} 1 & 2 & 3 \\ 1 & 1 & 1 \\ 2 & 3 & 4 \end{array} \right]$ and scale
successively the first, second or third row by a factor $\lambda = 3$, to get 
$$B= \left[ \begin{array}{ccc} 3 & 2 & 3 \\ 3 & 1 & 1 \\ 6 & 3 & 4 \end{array} \right], 
  C= \left[ \begin{array}{ccc} 1 & 6 & 3 \\ 1 & 3 & 1 \\ 2 & 9 & 4 \end{array} \right],
  D= \left[ \begin{array}{ccc} 1 & 2 & 9 \\ 1 & 1 & 3 \\ 2 & 3 & 12 \end{array} \right] \; .  $$
While $\det(A) = -2$ we have $\det(B)=-8, \det(C)=-2$ and $\det(D) =-4$.
A look at the eigenvalues $\sigma(A) = \{6.31662, -0.316625, 0\}$,
$\sigma(B) = \{8.89898, -0.898979, 0\}$, $\sigma(C) = \{8.24264, -0.242641, 0\}$ and
$\sigma(D) = \{14.2801, -0.28011, 0\}$ confirms how scaling of rows with the same factor
can be tossed around the spectrum all over the place. In this $3 \times 3$ example,
we can visualize the situation geometrically. The pseudo-determinant can be interpreted as the 
area of parallelogram in the image plane. Scaling rows deforms the triangle in 
different directions and it depends on the location of the triangle, how the area is scaled. \\

{\bf 2)}  The two  $3 \times 2$ matrices
$F=\left[ \begin{array}{cc} 1 & 1 \\
                            0 & 3 \\
                            2 & 1 \end{array} \right]$,
$G= \left[ \begin{array}{cc} 1 & 2 \\
                             1 & 1 \\
                             1 & 2 \end{array} \right]$ both have rank $2$.
If we append  $3$ times the first row to $F^T$ and 2 times
the first row to $G^T$, then the pseudo-determinant changes by a factor $1+2 \cdot 3 = 7$. 
Indeed, we have $\det(F^T F) = -9$ and $\det(A^T B) = -63$ with
$A=\left[ \begin{array}{ccc} 1 & 1 & 3 \\
                             0 & 3 & 0 \\
                             2 & 1 & 6 \end{array} \right]$ and 
$B=\left[ \begin{array}{ccc} 1 & 2 & 4 \\
                             1 & 1 & 2 \\
                             1 & 2 & 4 \end{array} \right]$. \\ 
If we append $3$ times the first row of $F^T$ to $F^T$ and $2$ times the {\bf second} row of $G^T$ 
to $G^T$ however, then the pseudo-determinant does not change because $1+\sum_i \mu_i \lambda_i=1$. 
Indeed we compute $\det(A^T B) = -9$ with
$A=\left[ \begin{array}{ccc} 1 & 1 & 3 \\
                             0 & 3 & 0 \\
                             2 & 1 & 6 \end{array} \right]$ and 
$B=\left[ \begin{array}{ccc} 1 & 2 & 2 \\
                             1 & 1 & 2 \\
                             1 & 2 & 2 \end{array} \right]$. 

{\bf 3)} This is an example of $4 \times 2$ matrices $F,G$, where the first matrix has 
rank $2$ and the second matrix has rank $1$. 
If $F=\left[ \begin{array}{cc} 1 & 1 \\
                                  2 & 1 \\
                                  3 & 1 \\
                                  4 & 1 \end{array} \right]$,
$G= \left[ \begin{array}{cc} 1 & 1 \\
                             2 & 2 \\
                             3 & 3 \\
                             4 & 4 \end{array} \right]$, 
Then $F^T G= \left[ \begin{array}{cc} 30 & 30  \\ 
                                      10 & 10  \end{array} \right]$
which has pseudo-determinant $40$.
If we multiply the second row of $F^T$ with a factor $2$,
we get 
 $F_1=\left[ \begin{array}{cc} 1 & 2 \\
                               2 & 2 \\
                               3 & 2 \\
                               4 & 2 \end{array} \right]$,
and  $F_1^T G= \left[ \begin{array}{cc} 30 & 30  \\ 
                                        20 & 20  \end{array} \right]$
which has pseudo-determinant $50$.
If we swap two rows of $F^T$, we get 
$F_2=\left[ \begin{array}{cc} 2 & 1 \\
                              2 & 2 \\
                              2 & 3 \\
                              2 & 4 \end{array} \right]$ and 
$F_2^T G = \left[ \begin{array}{cc} 10 & 10  \\ 
                                    30 & 30  \end{array} \right]$
which has the same pseudo-determinant $40$. If we subtract the second row of $F^T$
from the first row of $F^T$, we get 
$F_3=\left[ \begin{array}{cc} 0 & 1 \\
                              1 & 1 \\
                              2 & 1 \\
                              3 & 1 \end{array} \right]$ and 
$F_3^T G = \left[ \begin{array}{cc} 20 & 20 \\ 10 & 10 \end{array} \right]$ which 
has pseudo-determinant $30$.  \\
Finally, let's append twice the first row of $F^T$ to $F^T$ 
and three times the first row of $G^T$ to $G^T$ to get 
$A=\left[ \begin{array}{ccc} 1 & 1 & 2\\
                             2 & 1 & 4\\
                             3 & 1 & 6\\
                             4 & 1 & 8\end{array} \right]$ and
$B= \left[ \begin{array}{ccc} 1 & 1 & 3\\
                              2 & 2 & 6\\
                              3 & 3 & 9\\
                              4 & 4 & 12 \end{array} \right]$, so that 
$A^T B = \left[ \begin{array}{ccc} 30& 30& 90 \\
                                   10& 10& 30 \\
                                   60& 60& 180 \end{array} \right]$ which has rank $1$ and
$\Det(A^T B) = 220$.  For random $4 \times 2$ matrices $F,G$, scaling a row of $F^T$ by $\lambda$ 
changes the determinant of $F^T G$ by $\lambda$, swapping two rows of $F^T$ swaps
the determinant of $F^T G$ and subtracting a row of $F^T$ from an other row does not 
change the determinant. After appending $\lambda$ times
a rows to $F^T$ and $\mu$ times a row to $G^T$, we would get the pseudo-determinant scaled 
by a factor $1+\lambda \mu = 7$.  

\section{Remarks}

{\bf A)} Pseudo-determinants are useful also in infinite dimensions.
In infinite dimensions, one has regularized Fredholm determinants 
$\det(1+A) = \sum_{k=0}^{\infty} {\rm tr}(\Lambda^k A)$
for trace class operators $A$. If $A$ and $B$ are trace class, then they satisfy
the Cauchy-Binet formula $\det( (1+A)(1+B) ) = \det(1+A) \det(1+B)$ and Fredholm
determinants are continuous on trace class operators. (See e.g. \cite{SimonDeterminants,SimonTrace}).
Pseudo-determinants make an important appearance for zeta regularized 
Ray-Singer determinants \cite{LapidusFrankenhuijsen} of Laplacians for compact manifolds using the
Minakshisundaraman-Pleijel zeta function defined by the Laplacian. These
determinants are always pseudo-determinants because Laplacians on compact Riemannian manifolds 
always have a kernel. For the Dirac operator $A=i \frac{d}{dx}$ on the circle,
which has eigenvalues $n$ to the eigenvalues $e^{-i n}$ the Dirac zeta function
is $\zeta(s) = \sum_{n > 0} n^{-s}$. 
Since $\zeta'(0) = -1$, the circle has the 
Dirac Ray-Singer determinant ${\rm Det}(D)=e$. This must be seen as
a regularized pseudo-determinant because only the nonzero 
eigenvalues have made it into the zeta function.  \\

{\bf B)} Theorem~\ref{Theorem1}) will allow to 
give a combinatorial description of the pseudo-determinant of the Laplace-Beltrami 
operator $L = (d+d^*)^2$ acting on discrete differential forms of a finite simple graph.
The right hand side of Cauchy-Binet for pseudo determinants can be expressed in terms of trees. 
If $D=d+d^*$ is the Dirac matrix of a finite simple graph, 
then $L=D^2$ is the Laplace-Beltrami operator $L$ acting on discrete differential forms. 
It is a $v \times v$ matrix if $v$ is the total number of simplices in the graph.
The formula implies there that
$\Det(L)$ is the number of maximal complete trees in a double cover of $G$ branched at
a single vertex. Like the classical matrix tree theorem which gives an interpretation of the 
$\Det(L_0)$ where $L_0$ is the matrix Laplacian on function, the Dirac matrix theorem 
gives an interpretation of $\Det(L)$, where $L$ is the Laplace-Beltrami operator acting on 
discrete differential forms. Things are a bit more interesting in the Dirac case because trees
come with a positive or negative sign and unlike in the classical matrix tree theorem, 
$\det(D_P^2)$ does no more correspond to a single tree only. 
One has to look at a branched double cover of the graph so that we have to interpret
$\det(D_P) \det(D_P)$ a signed number of maximal labeled trees in a double cover of the 
simplex graph $\G$ branched at the base point. Now, $\Det(D)$ is divisible
by the number of vertices $|V|$ in the original graph $G$, if $G$ is connected. Note that for the 
scalar Laplacian $\L_0$ of the simplex graph $\G$, the classical matrix tree just counts the number of
trees in $\G$ but this is a different matrix than the form Laplacian $(d+d^*)^2$, evenso
they are both square $v \times v$ matrices. As $\G$ is a connected graph with $v$ vertices, the kernel of $\L_0$
is one dimensional and $\Det(\L_0)$ is divisible by $v$. The kernel of $L$ however has a dimension $\sum_i b_i$
where $b_k={\rm dim}(H^k(G))$ of $G$. In this linear algebra context it is worth mentioning that the
Euler-Poincar\'e identity $\sum_k (-1)^k b_k = \sum_k (-1)^k v_k$ for the Euler characteristic of $G$ or
the Hodge formula $b_k = {\rm dim}({\rm ker}(L_k)$ for the dimension of the $k$-harmonic forms are elementary
(see  \cite{knillmckeansinger,KnillILAS}). \\

{\bf C)} For circle graphs $C_n$, the scalar Laplacian $L=L_0$ satisfies 
$$  \Det(L)=\prod_{k=1}^{n-1} 4 \sin( \frac{k \pi}{n})^2 = n^2 \; . $$
We have looked at Birkhoff sums 
$\sum_{k=1}^{n-1} \log(\sin^2(\pi k \alpha/n))$ 
for Diophantine $\alpha$ in \cite{KT,knillcotangent} and
studied the roots of the zeta functions of these graphs which are entire functions
$$  \zeta_n(s) = \sum_{k=1}^{n-1} 2^{-s} \sin^{-s}(\pi \frac{k}{n}) $$ 
in \cite{KnillZeta}. \\

{\bf D)} Let $A$ be the adjacency matrix of a finite simple weighted graph.
The later means that real values $A_{ij} = A_{ji}$ called weights are assigned to the 
edges of the graph. A sub-matrix $P=P_{K,L}$ is obtained by restricting to a sub pattern.
If the square of $\det(P)^2$ is called the benefit of the sub-pattern and $\Det^2(A)$
the benefit of the matrix, then the Pythagorean pseudo-determinant formula~(\ref{Pythagoras}) tells that
the square $\Det^2(A)$ is the sum of the benefits of all $k \times k$ sub patterns, where
$k$ is the rank of $A$. The formula $\Det(1+A^2)$ is the sum of the benefits over all 
possible sub patterns. This picture confirms that both $\Det$ or $\det$ are interesting functionals. \\

{\bf E)} Lets look at a probability space of symmetric $n \times n$ matrices which take values in a finite
set. We can ask which matrices maximize or minimize the pseudo-determinant. Pseudo-determinants can be larger than
expected: for all $2 \times 2$ matrices taking values in $0,1$, the maximal determinant is $1$ while the 
maximal pseudo-determinant is $2$, obtained for the matrix where all entries are 1.
On a probability space $(\Omega,P)$ of matrices, where matrix entries have continuous distribution, 
it does of course not matter whether we take the determinant functional 
or pseudo-determinant functional because non-invertible matrices have zero probability. But we can ask
for which $n \times n$ matrices taking values in a finite set, 
the pseudo-determinant $\Det(A)$ is maximal.  \\

{\bf F)} We can look at the statistics of the pseudo-determinant on the Erd\"os-Renyi probability space of all 
finite simple graphs $G=(V,E)$ of order $|V|=n$ similar than for the Euler characteristic
or the dimension of graph in \cite{randomgraph}. 
Considering the pseudo-determinant functional on the subset of all connected graphs could be interesting. 
While the minimal pseudo-determinant is achieved for the complete graphs where $\Det(D(K_n)) = -n^{2^{n-1}-1}$,
it grows linearly for linear graphs $\Det(D(L_n)) = n (-1)^{n-1}$ and quadratically for cycle graphs
$\Det(D(C_n)) = n^2 (-1)^{n-1}$, where $D=D(C_n)$ is a $2n \times 2n$ matrix and $D^2$ is a block matrix containing
two copies $L_0,L_1$ of the $n \times n$ scalar Laplacian $L_0$ and $\Det(D^2)=\Det(L_0) \Det(L_1) = \Det^2(L_0)=n^4$. 
The complete graph with one added spike seems to lead to the largest pseudo 
determinant. We computed the pseudo-determinant for all graphs up to order $n=7$, where there are
1'866'256 connected graphs. It suggests that a limiting distribution of the random variable 
$X(G) = \log(|\Det(L(G)|)$ might exist on the probability
space of all connected graphs $G$ in the limit $n \to \infty$. \\

{\bf G)} The results were stated over fields of characteristic zero. Since multi-linear
algebra can be done over any field $F$, Theorem~\ref{Theorem2}) generalizes. 
Determinants over a commutative ring $K$ can
be characterized as an alternating $n$-linear function $D$ on $M(n,K)$
satisfying $D(1)=1$ (see e.g. \cite{HoffmanKunze,Dieudonne1}). As discussed, this does not apply
to pseudo-determinants. Is there an elegant axiomatic description of pseudo-determinants? I
asked this Fuzhen Zhang after his plenary talk in Providence who informed me that it is not
a {\bf generalized matrix function} in the sense of Marcus and Minc \cite{MarcusMinc65}, who 
introduced functions of the type 
$d(A) = \sum_{x \in H} \chi(x) \prod A_{i,x(i)}$, where $H$ is a subgroup of the symmetric group
and $\chi$ is a character on $H$. Indeed, the latter is continuous in $A$, 
while the pseudo determinant is not continuous as a function on matrices. 

\section*{Appendix}

We add some Mathematica routines which illustrate the results in this paper. First, we list
an implementation of the pseudo determinant, using the first non zero entry $p_k$ of the 
Characteristic polynomial multiplied by $(-1)^k$

\begin{tiny}
\lstset{language=Mathematica} \lstset{frameround=fttt}
\begin{lstlisting}[frame=single]
FirstNon0[s_]:=-(-1)^ArrayRules[s][[1,1,1]] ArrayRules[s][[1,2]];
PDet[A_]:=FirstNon0[CoefficientList[CharacteristicPolynomial[A,x],x]];
\end{lstlisting}
\end{tiny}

Next we look at the procedure ``PTrace" - the $P$ stands for Pauli trace -
which produces the pairing $\langle F,G \rangle_k$, which is the sum
over all products $\det(F_P) \det(G_P)$ where $P$ runs over $k \times k$
sub matrices. We compute the list of Pauli traces and a generating function $t(x)$ 
which will match the characteristic polynomial $p(x)$ of $F^T G$. 
There is a perfect match as our theorem tells that $p(x)=t(x)$. 

\begin{tiny}
\lstset{language=Mathematica} \lstset{frameround=fttt}
\begin{lstlisting}[frame=single]
PTrace[A_,B_,k_]:=Module[{U,V,m=Length[A],n=Length[A[[1]]],s={},t={}},
 U=Partition[Flatten[Subsets[Range[m],{k,k}]],k]; u=Length[U];
 V=Partition[Flatten[Subsets[Range[n],{k,k}]],k]; v=Length[V];
 Do[s=Append[s,Table[A[[U[[i,u]],V[[j,v]]]],{u,k},{v,k}]],{i,u},{j,v}];
 Do[t=Append[t,Table[B[[U[[i,u]],V[[j,v]]]],{u,k},{v,k}]],{i,u},{j,v}];
 Sum[Det[s[[l]]]*Det[t[[l]]],{l,Length[s]}]];
PTraces[A_,B_]:=(-x)^m+Sum[(-x)^(m-k) PTrace[A,B,k],{k,1,Min[n,m]}];
n=5; m=7; r=4; RandomMatrix[n_,m_]:=Table[RandomInteger[2r]-r,{n},{m}];
F=RandomMatrix[n,m]; G=RandomMatrix[n,m];
CharacteristicPolynomial[Transpose[F].G,x]
PTraces[F,G]
\end{lstlisting}
\end{tiny}

The code for experiments made to discover Theorem~\ref{Theorem2}) generates
matrices $F,G$ and then $A=F^T G$ and $B=F G^T$ for which both have zero 
determinant. This assures that we are in a situation, which can not be explained
by the classical Cauchy-Binet theorem. Then we compute both polynomials: 

\begin{tiny}
\lstset{language=Mathematica} \lstset{frameround=fttt}
\begin{lstlisting}[frame=single]
R:=RandomMatrix[n,m]; 
Shuffle:=Module[{},F=R;G=R;A=Transpose[F].G;B=F.Transpose[G]];
Shuffle; While[Abs[Det[A]]+Abs[Det[B]]>0,Shuffle]; 
{CharacteristicPolynomial[Transpose[F].G,x], PTraces[F,G]}
\end{lstlisting}
\end{tiny}

Here we compare the coefficients stated first in Theorem~\ref{Theorem2}):

\begin{tiny}
\lstset{language=Mathematica} \lstset{frameround=fttt}
\begin{lstlisting}[frame=single]
F=RandomMatrix[n,m]; G=RandomMatrix[n,m];
PTr[A_,B_]:=1+Sum[z^k PTrace[A,B,k],{k,1,Min[n,m]}];
{Det[IdentityMatrix[m]+z Transpose[F].G],PTr[F,G]}
\end{lstlisting}
\end{tiny}

And now we give examples of the identity in Corollary~\ref{determinanttheorem}):

\begin{tiny} 
\lstset{language=Mathematica} \lstset{frameround=fttt}
\begin{lstlisting}[frame=single]
A=RandomMatrix[n,m]; B=RandomMatrix[n,m]; 
{Det[IdentityMatrix[m] + Transpose[A].B], PTraces[A,B] /. x->-1}
\end{lstlisting} 
\end{tiny}

Finally, lets illustrate the multi-linear Lemma~\ref{multiplication}) by building 
two matrices $F,G$ and then comparing the matrix product of 
the exterior powers $\Lambda^k F, \Lambda^k G$ with the exterior power $\Lambda^k(FG)$ 
of the product $FG$: 

\begin{tiny} 
\lstset{language=Mathematica} \lstset{frameround=fttt}
\begin{lstlisting}[frame=single]
WedgeP[F_,k_]:=Module[{U,V,m=Length[F],n=Length[F[[1]]]}, 
 U=Partition[Flatten[Subsets[Range[m],{k,k}]],k];u=Length[U];
 V=Partition[Flatten[Subsets[Range[n],{k,k}]],k];v=Length[V];
 Table[Det[Table[F[[U[[i,u]],V[[j,v]]]],{u, k},{v, k}]],{i,u},{j,v}]];
F=RandomMatrix[n,m]; G=RandomMatrix[n,m]; 
Transpose[WedgeP[F,3]].WedgeP[G,3] - WedgeP[Transpose[F].G,3] 
\end{lstlisting} 
\end{tiny}


\vspace{12pt}
\bibliographystyle{plain}

\end{document}